\documentclass[3p,times]{elsarticle}

\usepackage[figuresright]{rotating}

\usepackage{amsmath}
\usepackage{amssymb}
\usepackage{amsthm}
\usepackage{amscd}
\usepackage{amsxtra}
\usepackage{verbatim}
\usepackage{xcolor}
\usepackage{color}
\usepackage{enumerate}
\usepackage{mathrsfs}
\usepackage{array,arydshln}
\usepackage{bm}
\usepackage{arydshln}
\usepackage{amsfonts}
\usepackage{mathrsfs}
\usepackage{bbm}
\usepackage{extarrows} 
\numberwithin{equation}{section}
\definecolor{dblue}{rgb}{0,0,0.45}
\definecolor{red}{rgb}{0.7,0,0}
\RequirePackage{geometry}
\newtheorem{theorem}{Theorem}[section]
\newtheorem{lemma}[theorem]{Lemma}
\newtheorem*{lemma*}{Lemma}

\newtheorem{proposition}[theorem]{Proposition}

\theoremstyle{definition}

\newtheorem{remark}[theorem]{Remark}

\theoremstyle{remark}

\begin{document}

\begin{frontmatter}

\title{Averaging principle for McKean-Vlasov SDEs driven by multiplicative fractional noise with highly oscillatory drift coefficient 
}

\author[mymainaddress,mysecondaddress]{Bin \textsc{Pei}}
\ead{binpei@nwpu.edu.cn}

\author[mymainaddress]{Lifang \textsc{Feng}}
\ead{FLF.fenglifang@outlook.com}

\author[mymainaddress]{Min \textsc{Han}}
\ead{mhan2019@hotmail.com}

\address[mymainaddress]{School of Mathematics and Statistics, Northwestern Polytechnical University,   Xi’an, 710072, China}
\address[mysecondaddress]{Chongqing Technology Innovation Center, Northwestern Polytechnical University,   Chongqing, 401120, China}

\begin{abstract}
In this paper, we study averaging principle for a class of McKean-Vlasov stochastic differential equations (SDEs) that contain multiplicative fractional noise with Hurst parameter  $H > $ 1/2 and highly oscillatory drift coefficient. Here the integral corresponding to fractional Brownian motion is the generalized Riemann-Stieltjes integral. Using Khasminskii's time discretization techniques, we prove that the solution of the original system strongly converges to the solution of averaging system as the times scale $ \epsilon $ gose to zero in the supremum- and H\"older-topologies which are sharpen existing ones in the classical Mckean-Vlasov SDEs framework.
 
\end{abstract}
\end{frontmatter}
\section{Introduction}\label{se-1}
The present paper focuses on the following McKean-Vlasov stochastic differential equations (SDEs) with highly oscillatory drift coefficient driven by multiplicative fractional noise in related path spaces, namely with supremum- and H\"older-topologies
\begin{equation}\label{mvsde1}
\mathrm{d}X_{t}^{\epsilon}=b({t}/{\epsilon},X_{t}^{\epsilon},\mathscr{L}_{X_{t}^{\epsilon}})\mathrm{d}t+\sigma(X_{t}^{\epsilon})\mathrm{d}B_{t}^{H},\quad X_{0}^{\epsilon}=x_0,\quad t\in[0,T]
\end{equation}
where the parameter $ 0 <\epsilon \ll 1 $, $ x_0 \in \mathbb{R} $ is arbitrary and non-random but fixed and the coefficients $ b:[0,T]\times\mathbb{R}\times\mathcal{P}_{2}(\mathbb{R})\rightarrow\mathbb{R} $ and $ \sigma:\mathbb{R}\rightarrow\mathbb{R} $ are measurable functions and $ \mathscr{L}_{X_{t}^{\epsilon}} $ is the law of $ X_{t}^{\epsilon} $. Here $ \mathcal{P}_{2}(\mathbb{R}) $ is the space of probability measures on $ \mathbb{R} $ with finite $ 2 $-th moment which will be introduced in Section \ref{se-2}. $\{B_t^H,t\geq 0\}$ is one dimensional fractional Brownian motion (FBM) with Hurst parameter $H>1/2$ which is a Gaussian centered process with the covariance function 
$$R_H(t,s)=\frac12 (t^{2H}+s^{2H}-|t-s|^{2H}).$$

The McKean-Vlasov SDEs, also known as distribution dependent SDEs or mean-field SDEs, whose evolution is determined by both the microcosmic location and the macrocosmic distribution of the particle, see e.g.  \cite{kac1956foundations}, \cite{mckean1966class}, \cite{braun1977vlasov}, can better describe many models than classical SDEs as their coefficients depend on the law of the solution.
Such kind of stochastic systems (\ref{mvsde1}) are of independent interest and appear widely in applications including granular materials dynamics, mean-field games, as well as complex networked systems, see e.g. \cite{bolley2013uniform}, \cite{lasry2007mean}, \cite{hu2021mean}.

Now, we remind the reader what an averaging principle is. Since the highly oscillating 
component,  it is relatively difficult to solve (\ref{mvsde1}). The main goal of the averaging principle
to find a simplified system which simulates and predicts the
evolution of the original system (\ref{mvsde1}) over a long time scale by averaging the highly oscillating drift coefficient under some suitable conditions. The history of averaging principle for deterministic systems is long
which can be traced back to the result by Krylov, Bogolyubov
and Mitropolsky, see e.g. \cite{krylov1950introduction}, \cite{bogoliubov1961asymptotic}. After that, \cite{khasminskii1968principle} established an averaging principle for the SDEs driven by Brownian motion (BM). Up to now, there have existed some kind of methods, such as the techniques of time discretization and Poisson equation, the weak convergence method,  studing averaging principle, see e.g. \cite{liu2010strong},
\cite{liu2020averaging}, \cite{pardoux2001poisson}, \cite{pei2020averaging}, \cite{sun2022strong} for SDEs,
and see e.g. \cite{cerrai2009averaging}, \cite{fu2011strong}, \cite{dong2018averaging}, \cite{gao2018averaging} for
stochastic partial differential equations (SPDEs).

In recent years there has been considerable research interest in averaging for Mckean-Vlasov stochastic (partial) differential equations S(P)DEs. \cite{rockner2021well} established the averaging principle for slow-fast Mckean-Vlasov SDEs by the techniques of time discretization and Poisson equation. \cite{hong2022strong} investigated the strong convergence rate of averaging principle for slow-fast Mckean-Vlasov
SPDEs based on the variational approach and the technique of time discretization. \cite{cheng2022strong} studied averaging principle for distribution dependent SDEs with localized $L^p$ drift using Zvonkin's transformation and estimates for Kolmogorov
equations. \cite{shen2022averaging} obtained the strong convergence without a
rate for distribution dependent SDEs with highly oscillating component driven by FBM and
standard BM, which requires that the FBM-term should be additive case.

However, the aforementioned references all focused on the Mckean-Vlasov S(P)DEs with addictive noise or multiplicative white noise. Up to now, there are no work concentrating on averaging for Mckean-Vlasov SDEs driven by multiplicative fractional noise. In this work, we aim to close this gap. It is known that the FBMs are not semimartingales. Therefore, the beautiful classical stochastic analysis is not applicable to fractional noises for $H\neq 1/2$. It is a non-trivial task to extent the results in the classical stochastic analysis to these multiplicative fractional noises while one can use Wiener integral for the addictive fractional noise because the diffusion term is a dererministic function.  Note that the diffusion term of Mckean-Vlasov SDEs  in this paper is state variables-dependent, based on Riemann Stieltjes integral framework, we cannot
use Gronwall's lemma or generalized Gronwall’s lemma directly to prove the convergence of $X^\epsilon$ to $X$ as in \cite{mishura2011stochastic}, \cite{mishura2012mixed}. So, we will use the $\lambda$-equivalent H\"older norm (see, Section 4.3 ) to overcome this problem.

From the above motivations, we consider the strong convergence of averaging principle
for Mckean-Vlasov SDEs driven by multiplicative fractional noise in the present paper. The problem is solved by the fractional approach and Khasminskii type averaging principle efficiently. Moreover, our averaging result in the supremum- and H\"older-topologies sharpen existing ones in the classical Mckean-Vlasov S(P)DEs framework.

The paper is organized as follows. Section \ref{se-2} presents some necessary notations and assumptions. Stochastic averaging principles for such McKean-Vlasov SDEs are then established in Section \ref{se-3}. Note that $ C $ and $ C_{\mathrm{x}} $ denote some positive constants which may change from line to line throughout this paper, where $ \mathrm{x} $ is one or more than one parameter and $ C_{\mathrm{x}} $ is used to emphasize that the constant depends on the corresponding parameter, for example, $ C_{\alpha,\beta,\gamma,T,|x_0|} $ depends on $ \alpha,\beta,\gamma,T $ and $ |x_0| $.
\section{Preliminaries}\label{se-2}
In this section, we will recall some basic facts on definitions and properties of the fractional caculus. For more
details, we refer to \cite{guerra2008stochastic} and \cite{rascanu2002differential}.  Firstly, we now introduce some necessary spaces and norms. In what follows of the rest of this section , let $a,b\in \mathbb{R}, a<b$.
For $\gamma \in(0,1)$, let $C^\gamma((a,b),\mathbb{R})$ be the space of $\gamma$-H\"{o}lder continuous functions $f:[a,b]\rightarrow\mathbb{R}$, equipped with the the norm $$
\|f\|_{\gamma,a,b} =\|f\|_{\infty,a,b}+\||f\||_{\gamma,a,b}$$
with
$$\|f\|_{\infty,a,b}:= \sup _{t \in[a, b]}|f(t)|,\quad\||f\||_{\gamma,a,b}=\sup_{a \leq s < t \leq b}\frac{|f(t)-f(s)|}{(t-s)^{\gamma}}.$$
  For simplify, let $
  \|f\|_{\beta} :=\|f\|_{\beta,0,T},\|f\|_{\infty}:=\|f\|_{\infty,0,T}$ and $\||f\||_{\beta}:=\||f\||_{\beta,0,T} $.

The following proposition provides an explicit expression for the integral $ \int_{a}^{b} f \mathrm{d} g $ when $f\in C^\gamma((a,b),\mathbb{R})$ and $g\in C^\beta((a,b),\mathbb{R})$ with $\beta+\gamma>1, \beta, \gamma \in (0,1)$  in terms of fractional derivatives, see \cite{zahle1998integration}.
\begin{proposition}
{\rm (Remark 4.1 in \citeauthor{rascanu2002differential}, \citeyear{rascanu2002differential}).} Suppose that $f\in C^\gamma((a,b),\mathbb{R})$ and $g\in C^\beta((a,b),\mathbb{R})$ with $\beta+\gamma>1, \beta, \gamma \in (0,1)$. Let $\alpha \in (0,1)$, $\gamma > \alpha$ and $\beta > 1- \alpha$. Then the Riemann Stieltjes integral $ \int_{a}^{b} f \mathrm{d} g $ exists and it can be expressed as
\begin{equation}\label{grsinteg}
	\int_{a}^{b} f \mathrm{d} g =(-1)^{\alpha} \int_{a}^{b} D_{a+}^{\alpha}f(t) D_{b-}^{1-\alpha}g_{b-}(t)\mathrm{d}t
\end{equation}
where $ g_{b-}(t)=g(t)-g(b) $ and for $ a\leq t\leq b $ the Weyl derivatives of $ f $ are defined by formulas
\begin{eqnarray*}	
    D_{a+}^\alpha f(t)
    &=\frac{1}{\Gamma(1-\alpha)}\bigg(\frac{f(t)}{(t-a)^\alpha}+\alpha\int_a^t\frac{f(t)-f(s)}{(t-s)^{\alpha+1}}\mathrm{d}s\bigg)\label{wleft},\\
    D_{b-}^{\alpha} f (t)
    &=\frac{(-1)^{\alpha}}{\Gamma(1-\alpha)}\bigg(\frac{f(t)}{(b-t)^{\alpha}}+\alpha\int_t^b\frac{f(t)-f(s)}{(s-t)^{\alpha+1}}\mathrm{d}s\bigg)\label{wright}
\end{eqnarray*}
and $ \Gamma $ denotes the Gamma function. 
\end{proposition}
\begin{lemma}\label{intfbm}{\rm (Theorem 2 in \citeauthor{hu2007differential}, \citeyear{hu2007differential})}
	Suppose that $f\in C^\gamma([0, T],\mathbb{R})$ and $g\in C^\beta([0, T],\mathbb{R})$ with $\beta+\gamma>1$ and $1>\gamma>\alpha,\,1>\beta>1-\alpha$, for all $ s, t \in[0, T] $, one has
	\begin{eqnarray*}
		\bigg|\int_{s}^{t}f(r)\mathrm{d}g(r)\bigg|
		\leq
		C_{\alpha,\beta,\gamma,T}\|f\|_{\gamma,s,t}\||g\||_{\beta}(t-s)^{\beta}.
	\end{eqnarray*}
\end{lemma}
\begin{proof}
	By using the fractional integration given in (\ref{grsinteg}), for all $ s, t \in[0, T] $, we have
	\begin{eqnarray*}
		\bigg|\int_{s}^{t}f(r)\mathrm{d}g(r)\bigg|
		&\leq&
		\int_{s}^{t}\big|D_{s+}^{\alpha}f(r)\big|\cdot \big| D_{t-}^{1-\alpha}g_{t-}(r)\big|\mathrm{d}r\cr
&\leq & \int_{s}^{t} C_{\alpha,\gamma,T}\|f\|_{\gamma,s,t}(r-s)^{-\alpha} C_{\alpha,\beta}\||g\||_{\beta}(t-r)^{\alpha+\beta-1}\mathrm{d}r\cr
&\leq &C_{\alpha,\beta,\gamma,T}\|f\|_{\gamma,s,t}\||g\||_{\beta} \int_{s}^{t} (r-s)^{-\alpha} (t-r)^{\alpha+\beta-1}\mathrm{d}r\cr
&\leq&
		C_{\alpha,\beta,\gamma,T}\|f\|_{\gamma,s,t}\||g\||_{\beta}(t-s)^{\beta}.
	\end{eqnarray*}
This completes the proof.
\end{proof}

\begin{remark}\label{fbmito1}{\rm (Lemma 7.5 in \citeauthor{rascanu2002differential}, \citeyear{rascanu2002differential})}
	The trajectories of $B^H$ are locally $\beta$-H\"{o}lder
		continuous a.s. for all  $\beta \in (0,H)$ and $\||B^H\||_{\beta}$ has moments of all order.
\end{remark}

\begin{remark}\label{fbmito}
	Suppose that $f\in C^\gamma([0, T],\mathbb{R})$ and $B^H\in C^\beta([0, T],\mathbb{R})$ with $\beta+\gamma>1$ and $1>\gamma>\alpha,\, 1>\beta>1-\alpha$, for all $ s, t \in[0, T] $, one has
	\begin{eqnarray*}
		\bigg|\int_{s}^{t}f(r)\mathrm{d}B^H_r\bigg|
		\leq
		C_{\alpha,\beta,\gamma,T}\|f\|_{\gamma,s,t}\||B^H\||_{\beta}(t-s)^{\beta}, \,\,\, {\rm a.s.}
	\end{eqnarray*}
\end{remark}

\begin{lemma}\label{lemma2.5}
	For any positive constants $ a $, $ d $, if $ a, d<1 $, one has
	\begin{eqnarray*}
		\int_{s}^{t}(r-s)^{-a}(t-r)^{-d}\mathrm{d}r
		\leq
		(t-s)^{1-a-d}B(1-a,1-d)
	\end{eqnarray*}
    where $ s\in(0,t) $ and $ \mathit{B} $ is the Beta function.
\end{lemma}
\begin{proof}
	By a change of variable $ y=(r-s)/(t-s) $, we have
	\begin{eqnarray*}
		\int_{s}^{t}(r-s)^{-a}(t-r)^{-d}\mathrm{d}r
		&=&
		\int_{0}^{1}(y(t-s))^{-a}(t-s-y(t-s))^{-d}(t-s)\mathrm{d}y\cr
		&=&
		(t-s)^{1-a-d}\int_{0}^{1}y^{-a}(1-y)^{-d}\mathrm{d}y\cr
		&=&
		(t-s)^{1-a-d}B(1-a,1-d).
	\end{eqnarray*}
This completes the proof.
\end{proof}

Let $ \mathcal{P}(\mathbb{R}) $ be the collection of all probability measures on $ \mathbb{R} $,	and $ \mathcal{P}_{2}(\mathbb{R}) $ be the space of probability measures on $ \mathbb{R} $ with finite $ 2 $-th moment, i.e.,
$$  \mathcal{P}_{2}(\mathbb{R})=\bigg\{\mu\in\mathcal{P}(\mathbb{R}):\;\mu(|\cdot|^2):=\int_{\mathbb{R}}|x|^2\mu(\mathrm{d}x)<\infty\bigg\}.  $$

 We define the $ L^{2} $-Wasserstein distance on $ \mathcal{P}_{2}(\mathbb{R}) $ by $$
\mathbb{W}_{2}\left(\mu_{1}, \mu_{2}\right):=\inf _{\pi \in \mathcal{C}_{\mu_{1}, \mu_{2}}}\left(\int_{\mathbb{R} \times \mathbb{R}}|x-y|^{2} \pi(d x, d y)\right)^{1 / 2}, \quad \mu_{1}, \mu_{2} \in \mathcal{P}_{2}(\mathbb{R})
$$where $\mathcal{C}_{\mu_{1}, \mu_{2}}$ is the set of probability measures on $ \mathbb{R} \times \mathbb{R} $ with marginals $\mu_{1}$ and $\mu_{2}$. It is well-known that $ (\mathcal{P}_{2}(\mathbb{R}),\mathbb{W}_{2}) $ is a Polish space.

Note that for any  $x \in \mathbb{R}$, the Dirac measure $\delta_{x}$  belongs to  $\mathcal{P}_{2}(\mathbb{R})$, specially $ \delta_0 $ is the Dirac measure at point 0 and if  $\mu_{1}=\mathscr{L}_X, \mu_{2}=\mathscr{L}_Y$  are the corresponding distributions of random variables  $X$  and $Y$  respectively, then$$
\mathbb{W}_{2}^{2}(\mu_{1}, \mu_{2}) \leq \int_{\mathbb{R}\times \mathbb{R}}|x-y|^{2} \mathscr{L}_{(X, Y)}(d x, d y)=\mathbb{E}[|X-Y|^{2}]$$in which  $\mathscr{L}_{(X, Y)}$  represents the joint distribution of the random pair $ (X, Y)$. Then for arbitrarily fixed  $T>0$ , let  $C([0, T] ; \mathbb{R})$  be the Banach space of all $ \mathbb{R}$-valued continuous functions on
$ [0, T]$, endowing with the supremum norm. Furthermore, we let $ L^{2}(\Omega ; C([0, T] ; \mathbb{R}))$  be the totality of $ C([0, T] ; \mathbb{R})$-valued random variables  $X $ satisfying  $\mathbb{E}[\sup _{0 \leq t \leq T}|X(t)|^{2}]<\infty$. Then,  $L^{2}(\Omega ; C([0, T] ; \mathbb{R}))$  is a Banach space under the norm
$$
\|X\|_{L^{2}}:=\Big(\mathbb{E}\Big[\sup _{0 \leq t \leq T}|X(t)|^{2}\Big]\Big)^{1/2}.$$

\section{Assumptions and main result}\label{se-3}
\subsection{Assumptions} 
To derive a unique solution to (\ref{mvsde1}), we first introduce assumptions on the coefficients $b$ and $\sigma$ such that
	\begin{enumerate}[(H1)] 
		\item There exists a constant $L_b>0$, such that for any $ t\in [0, T] $, $x_{1}, x_{2}\in \mathbb{R} $ and $\mu_{1}, \mu_{2}\in \mathcal{P}_{2}(\mathbb{R}) $,
		\begin{align*}
		|b(t,x_{1},\mu_{1})-b(t,x_{2},\mu_{2})|
		\leq L_b \big(|x_{1}-x_{2}|+\mathbb{W}_{2}(\mu_{1},\mu_{2})\big).
		\end{align*}
		Moreover, $ b $ is bounded by a positive constant $ M_b $, i.e., $$ \sup_{(t,x,\mu)\in[0,T]\times\mathbb{R}\times\mathcal{P}_{2}(\mathbb{R})}|b(t,x,\mu)|\leq M_b. $$
	\end{enumerate}
    \begin{enumerate}[(H2)] 
    	\item There exist constants $M_\sigma>K_{\sigma} >0$ and $L_{\sigma}>0$ such that for any $x, x_{1}, x_{2}\in \mathbb{R} $
    	\begin{align*}
    	|\sigma(x_1)-\sigma(x_2)| \leq L_{\sigma} |x_{1}-x_{2}| \, {\rm and}\,\,
    	K_{\sigma} \leq|\sigma(x)|\leq M_\sigma.
    	\end{align*}
    \end{enumerate}

Under assumptions (H1) and (H2) above, one can deduce from Theorem 3.3 in \cite{fan2022distribution} that the system (\ref{mvsde1}) admits a unique solution via a Lamperti transform.
\begin{lemma}\label{lemma1}
	Suppose that {\rm (H1)} and {\rm (H2)} hold and $ 1/2 < H <1 $, then {\rm (\ref{mvsde1})} has a unique solution $ X\in L^{2}(\Omega ; C([0, T] ; \mathbb{R}))$.
\end{lemma}

In order to establish the averaging principle, besides conditions (H1) and (H2), we further assume:
	\begin{enumerate}
		\item [(H3)] $ b $ is Lipschitz continuous respect to $ t $, i.e.,  there exists a positive constant $ L^\prime_b $, such that for any $ t\in [0, T] $, $x\in \mathbb{R} $ and $\mu \in \mathcal{P}_{2}(\mathbb{R}) $,
		\begin{align*}
		|b(t_1,x,\mu)-b(t_2,x,\mu)|
		\leq L^\prime_b|t_{1}-t_{2}|.
		\end{align*}
		\item [(H4)] The function $ \sigma(x) $ is of class $C^1(\mathbb{R})$. There exists a constant $M'_\sigma>0$ such that for any $ x, x_{1}, x_{2}\in \mathbb{R} $,
			\begin{align*}
			|\nabla \sigma(x_1)-\nabla \sigma(x_2)| \leq M'_\sigma |x_{1}-x_{2}| \, {\rm and}\,\,
			|\nabla \sigma(x)|\leq M'_\sigma
			\end{align*}
hold.
			Here, $ \nabla $ is the standard gradient operator on $\mathbb{R}$.
\item [(H5)] There exist a bounded positive function $ \varphi:\mathbb{R}^{+}\rightarrow\mathbb{R}^{+}$ and a measurable function $ \bar{b}:\mathbb{R}\times\mathcal{P}_{2}(\mathbb{R})\rightarrow\mathbb{R}$, such that for any $ x \in \mathbb{R} $, $ \mu\in \mathcal{P}_{2}(\mathbb{R}) $ it holds that
	\begin{eqnarray*}
		\sup_{t\geq 0}\Bigg|\frac{1}{T}\int_{t}^{t+T}(b(s, x, \mu)-\bar{b}(x, \mu))\mathrm{d}s\Bigg|
		\leq 
		\varphi(T)\big(1+|x|+\mu(|\cdot|^2)\big)
	\end{eqnarray*}
	where $ \varphi(T) $ satisfies $ \lim_{T\rightarrow\infty}\varphi(T)=0 $.
	\end{enumerate}

\begin{remark}\label{remark3.3}
	It follows from the conditions {\rm (H1)} and {\rm (H5)} that $ \bar{b} $ satisfies 
	\begin{align*}
	|\bar{b}(x,\mu)|& \leq L_{\bar{b}}, \\
	|\bar{b}(x_{1},\mu_{1})-\bar{b}(x_{2},\mu_{2})|
	&\leq L_{\bar{b}} \big( |x_{1}-x_{2}|+\mathbb{W}_{2}(\mu_{1},\mu_{2})\big)
	\end{align*}
	for any $ x, x_{1}, x_{2}\in \mathbb{R} $ and $ \mu, \mu_{1}, \mu_{2}\in \mathcal{P}_{2}(\mathbb{R}) $, where $ L_{\bar{b}}  $ is a positive constant.
\end{remark}
\begin{proof}
	We have that
	\begin{eqnarray*}
	|\bar{b}(x,\mu)| 
	&\leq&
	\Bigg|\frac{1}{T}\int_{0}^{T}(b(s, x, \mu)-\bar{b}(x, \mu))\mathrm{d}s\Bigg|+\Bigg|\frac{1}{T}\int_{0}^{T}b(s, x, \mu)\mathrm{d}s\Bigg|\cr
    &\leq&
    \varphi(T)\big(1+|x|+\mu(|\cdot|^2)\big)+M_{b} 
	\end{eqnarray*}
and
    \begin{eqnarray*}
    	|\bar{b}(x_{1},\mu_{1})-\bar{b}(x_{2},\mu_{2})| 
    	&\leq&
    	\Bigg|\frac{1}{T}\int_{0}^{T}(b(s, x_1, \mu_1)-\bar{b}(x_1, \mu_1))\mathrm{d}s\Bigg|
    	+\Bigg|\frac{1}{T}\int_{0}^{T}(b(s, x_2, \mu_2)-\bar{b}(x_2, \mu_2))\mathrm{d}s\Bigg|\cr
    	& &
    	+\Bigg|\frac{1}{T}\int_{0}^{T}(b(s, x_1, \mu_1)-b(s, x_2, \mu_2))\mathrm{d}s\Bigg|\cr
    	&\leq&
    	\varphi(T)\big(1+|x_1|+|x_2|+\mu_1(|\cdot|^2)+\mu_2(|\cdot|^2)\big)+L_b\big(|x_{1}-x_{2}|+\mathbb{W}_{2}(\mu_{1},\mu_{2})\big).
    \end{eqnarray*}

Taking $ T\rightarrow\infty $, there exist a constant $ L_{\bar{b}}>0$ such that Lemma \ref{remark3.3} holds. 
\end{proof}
\begin{remark}\label{remark3.4}
	Noting that
	\begin{eqnarray*}
		\sup_{t\geq 0}\Bigg|\frac{1}{T}\int_{t}^{t+T}(b(s, x, \mu)-\bar{b}(x, \mu))\mathrm{d}s\Bigg|
		\leq 
		\sup_{t\geq 0}\frac{1}{T}\int_{t}^{t+T}|b(s, x, \mu)-\bar{b}(x, \mu)|\mathrm{d}s.
	\end{eqnarray*}
This shows that the averaging condition (H5) is weaker than the following averaging condition
	\begin{eqnarray*}
		\sup_{t\geq 0}\frac{1}{T}\int_{t}^{t+T}|b(s, x, \mu)-\bar{b}(x, \mu)|\mathrm{d}s
		\leq 
		\varphi(T)\big(1+|x|+\mu(|\cdot|^2)\big).
	\end{eqnarray*}
\end{remark}
\subsection{Main result}
Now, we define the averaged equation:
\begin{equation}\label{avermvsde}
\mathrm{d}\bar{X}_{t}=\bar{b}(\bar{X}_{t},\mathscr{L}_{\bar{X}_{t}})\mathrm{d}t+\sigma(\bar{X}_{t})\mathrm{d}B_{t}^{H},\quad \bar{X}_{0}=x_0
\end{equation}
where $\bar b$ has been given in (H5) and using Theorem 3.3 in \cite{fan2022distribution} again, we have the unique solution result to (\ref{avermvsde}).

\begin{lemma}\label{lemma4}
	Suppose that {\rm (H1)}-{\rm (H5)} hold, then Eq. (\ref{avermvsde}) has a unique solution $ \bar{X} \in L^{2}(\Omega ; C([0, T] ; \mathbb{R})) $.
\end{lemma}

\begin{theorem}\label{thm1}
	Suppose that {\rm(H1)-(H5)} hold, then we obtain
	\begin{equation*}
	\lim_{\epsilon\rightarrow0}\mathbb{E}\big[\|X^{\epsilon}-\bar{X}\|_{\gamma}^2\big]=0.
	\end{equation*}
\end{theorem}
The proof of Theorem \ref{thm1} will be given in Section \ref{se-5}.
\begin{remark}
The averaging principle result (Theorem \ref{thm1}) is also applicable to the following system
\begin{align}\label{epsd}
\begin{split}
\mathrm d X_t=\epsilon b(t,X_t,\mathscr{L}_{X_t})\mathrm dt+\epsilon^H \sigma(X_t)\mathrm d B^H_t.
\end{split}
\end{align}

Let $t\mapsto{\frac{t}{\epsilon}}$, define $Y_{t}^{\epsilon}:=X_{t/\epsilon}$ and $B^{\epsilon, H}_{t}:=\epsilon^H B^H_{t/\epsilon}$ for all $t\in\mathbb{R}^{+}$ we rewrite (\ref{epsd}) as
\begin{align}
\begin{split}
&\mathrm{d}Y_{t}^{\epsilon}=b(t/\epsilon,Y_{t}^{\epsilon},\mathscr{L}_{Y_{t}^{\epsilon}})\mathrm{d}t+\sigma(Y_{t}^{\epsilon})\mathrm{d}B^{\epsilon, H}_{t}.
\end{split}
\end{align}

Then~we~can~consider~the~following system
\begin{align}
\begin{split}
\mathrm{d}\tilde{X}_{t}^{\epsilon}=b(t/\epsilon,\tilde{X}_{t}^{\epsilon},\mathscr{L}_{\tilde{X}_{t}^{\epsilon}})\mathrm{d}t+\sigma(\tilde{X}_{t}^{\epsilon})\mathrm{d}B^H_{t}.
\end{split}
\end{align}
\end{remark}

\section{The proof of main result}\label{se-5}
\subsection{Some a-prior estimate of the solution}
\begin{lemma}\label{xbounded}
	Suppose that {\rm(H1)-(H5)} hold. Then, for $ t \in [0, T] $, we have 
	$$  \|X^{\epsilon}\|_{\gamma}+\|\bar{X}\|_{\gamma} \leq C_{\alpha,\beta,\gamma,T,|x_0|}\big((1+\||B^H\||_{\beta})\vee(1+\||B^H\||_{\beta})^{\frac{1}{\gamma}}\big),\, \, {\rm a.s.} $$
\end{lemma}
\begin{proof}
Like the proof of Theorem 2.2 in \cite{hu2007differential} and Exercise 4.5 in \cite{friz2020course}, for any $ 0\leq s < t\leq T $, we have 
	\begin{eqnarray*}
	\||X^{\epsilon}\||_{\gamma,s,t}
	\leq
	C_{\alpha,\beta,\gamma,T}(1+\||B^H\||_{\beta})\big(1+\||X^{\epsilon}\||_{\gamma,s,t}(t-s)^{\gamma}\big).
	\end{eqnarray*}
	
	Suppose that $ \Delta $ satisfies $
	\Delta 
	=
	(2 C_{\alpha,\beta,\gamma,T}(1+\||B^H\||_{\beta})^{-\frac{1}{\gamma}}.$
	Then, for all $ s $ and $ t $ such that $ t - s \leq \Delta $ we have
	\begin{eqnarray}\label{xbeta2}
	\||X^{\epsilon}\||_{\gamma,s,t}
	\leq
	2 C_{\alpha,\beta,\gamma,T}(1+\||B^H\||_{\beta}).
	\end{eqnarray}
	
	Therefore we can obtain
	\begin{eqnarray}\label{xinfty1}
	\|X^{\epsilon}\|_{\infty,s,t}
	\leq
	|X^{\epsilon}_s|+\||X^{\epsilon}\||_{\gamma,s,t}(t-s)^{\gamma}
	\leq
	|X^{\epsilon}_s|+2 C_{\alpha,\beta,\gamma,T}(1+\||B^H\||_{\beta})\Delta^{\gamma}.
	\end{eqnarray}
	
If $ \Delta\geq T $, from (\ref{xbeta2}) and (\ref{xinfty1}), we obtain the estimate
	\begin{eqnarray}\label{xdt1}
	\|X^{\epsilon}\|_{\infty,s,t}
	\leq
	|x_0|+2 C_{\alpha,\beta,\gamma,T}(1+\||B^H\||_{\beta})T^{\gamma} \, {\rm and}\,\,
	\||X^{\epsilon}\||_{\gamma}
	\leq
	2 C_{\alpha,\beta,\gamma,T}(1+\||B^H\||_{\beta}).
	\end{eqnarray}
	
	While if $ \Delta\leq T $, then from (\ref{xinfty1}) we get
	\begin{eqnarray}\label{xinfty2}
	\|X^{\epsilon}\|_{\infty,s,t}
	\leq
	|X^{\epsilon}_s|+2 C_{\alpha,\beta,\gamma,T}(1+\||B^H\||_{\beta})(2 C_{\alpha,\beta,\gamma,T}(1+\||B^H\||_{\beta}))^{-{\frac{1}{\gamma}}\cdot\gamma}
	\leq
	|X^{\epsilon}_s|+1.
	\end{eqnarray}
	
Divide the interval $ [0, T] $ into $ n = [\frac{T }{\Delta}] + 1 $ subintervals, and use the estimate (\ref{xinfty2}) in every interval, we obtain
	\begin{eqnarray}\label{xsup2}
	\|X^{\epsilon}\|_{\infty}
	\leq
	|x_0|+n
	\leq
	|x_0|+T\Delta^{-1}+1
	\leq
	|x_0|+2T\big(2 C_{\alpha,\beta,\gamma,T}(1+\||B^H\||_{\beta})\big)^{\frac{1}{\gamma}}
	\leq
	|x_0|+2^{\frac{1+\gamma}{\gamma}}T C_{\alpha,\beta,\gamma,T}^{\frac{1}{\gamma}}(1+\||B^H\||_{\beta})^{\frac{1}{\gamma}}
	\end{eqnarray}
	and from (\ref{xbeta2}), we know that when $ t-s \leq \Delta $, then $ \||X^{\epsilon}\||_{\gamma,s,t}\leq 2 C_{\alpha,\beta,\gamma,T}(1+\||B^H\||_{\beta})$, when $ t-s \geq \Delta $, define $ t_{i}=(s+i\Delta)\wedge t $, for $ i=0,1,\dots,N $, noting that $ t_{N}=t $ for $ N\geq |t-s|/\Delta $ and also $ t_{i+1}-t_{i}\leq \Delta $ for all $ i $, then 
	\begin{eqnarray*}
	|X^{\epsilon}_t-X^{\epsilon}_s|
	\leq
	\sum_{0\leq i<|t-s|/\Delta}|X^{\epsilon}_{t_{i+1}}-X^{\epsilon}_{t_{i}}|
	\leq
	\big(|t-s|/\Delta+1\big)\Delta^{\gamma}
	=
	\Delta^{\gamma-1}(\Delta+|t-s|)
	\leq
	2\Delta^{\gamma-1}|t-s|.
	\end{eqnarray*}
	
	So we have
	\begin{eqnarray}\label{xbeta3}
	\||X^{\epsilon}\||_{\gamma,s,t}
	&\leq&
	2 C_{\alpha,\beta,\gamma,T}(1+\||B^H\||_{\beta})(1\vee	2\Delta^{\gamma-1})
	=
	\Big\{2 C_{\alpha,\beta,\gamma,T}(1+\||B^H\||_{\beta})\vee	2\Delta^{-1}\Big\}\cr
	&=&
	\Big\{2 C_{\alpha,\beta,\gamma,T}(1+\||B^H\||_{\beta})\vee	2\big(2 C_{\alpha,\beta,\gamma,T}(1+\||B^H\||_{\beta})\big)^{\frac{1}{\gamma}}\Big\}.
	\end{eqnarray}
	
	Thus, from (\ref{xdt1}), (\ref{xsup2}) and (\ref{xbeta3}), we get the desire estimate.
	
	Using similar techniques, we can prove
	$$ \|\bar{X}\|_{\gamma} \leq C_{\alpha,\beta,\gamma,T,|x_0|}\big((1+\||B^H\||_{\beta})\vee(1+\||B^H\||_{\beta})^{\frac{1}{\gamma}}\big),\, \, {\rm a.s.} $$
Here we omit the proof.
\end{proof}
\begin{lemma}\label{x-xh}
	Suppose that {\rm(H1)}-{\rm(H5)} hold. Then, if $ 0\leq t\leq t+h\leq T $, and $ h\in(0,1) $, we have
	$$ |X^{\epsilon}_{t+h}-X^{\epsilon}_{t}|\leq C_{\alpha,\beta,\gamma,T}(1+\||B^H\||_{\beta})(1+\|X^{\epsilon}\|_{\gamma})h^{\beta},\, \, {\rm a.s.} $$
\end{lemma}
\begin{proof}
	From (\ref{mvsde1}), by (H1)-(H5), H\"{o}lder inequality and Remark \ref{fbmito}, we have
	\begin{eqnarray*}
		|X^{\epsilon}_{t+h}-X^{\epsilon}_{t}|
		&\leq&
		\bigg|\int_{t}^{t+h}b({s}/{\epsilon},X^{\epsilon}_{s},\mathscr{L}_{X^{\epsilon}_{s}})\mathrm{d}s\bigg|
		+\bigg|\int_{t}^{t+h}\sigma(X^{\epsilon}_{s})\mathrm{d}B^{H}_{s}\bigg|\cr
		&\leq&
		\int_{t}^{t+h}|b({s}/{\epsilon},X^{\epsilon}_{s},\mathscr{L}_{X^{\epsilon}_{s}})|\mathrm{d}s
		+
		C_{\alpha,\beta,\gamma,T}\||B^H\||_{\beta}\Big(\sup_{t\leq r\leq t+h}|\sigma(X^{\epsilon}_{r})|+\sup_{t\leq u< r\leq t+h}\frac{|\sigma(X^{\epsilon}_{r})-\sigma(X^{\epsilon}_{u})|}{(r-u)^{\gamma}}\Big)h^{\beta}\cr
		&\leq& C_{\alpha,\beta,\gamma,T}(1+\||B^H\||_{\beta})\big(1+\|X^{\epsilon}\|_{\gamma}\big)h^{\beta}.
	\end{eqnarray*}
This completes the proof.
\end{proof}

\subsection{The proof of Theorem \ref{thm1}}
For each $R>1$, we define the following stopping times $\tau_R$ such that
\begin{eqnarray}\label{stoptime}
\tau_R :=\inf \{t\geq 0:\||B^{H}\||_{\beta,0,t}> R \} \wedge T.
\end{eqnarray}

	Firstly, we have
	\begin{eqnarray}\label{x-xbarstop}
	\mathbb{E}\big[\|X^{\epsilon}-\bar{X}\|_{\gamma}^2\big] 
	\leq
	\mathbb{E}\big[\|X^{\epsilon}-\bar{X}\|_{\gamma}^2\mathbf{1}_{\{\tau_R \geq T\}}\big]+\mathbb{E}\big[\|X^{\epsilon}-\bar{X}\|_{\gamma}^2\mathbf{1}_{\{\tau_R < T\}}\big]
	\end{eqnarray}
where $\mathbf{1}_{\cdot}$ is an indicator function. For the first supremum in the right-hand side of inequality (\ref{x-xbarstop}), denote $D :=\{\||B^{H}\||_{\beta} \leq R\} $. Now for $\lambda \geq 1$ a equivalent norm of $ C^\gamma([0,T],\mathbb{R}) $ with $\gamma\in (0,1)$ is defined by 
	$$\|f\|_{\gamma,\lambda,0,T} :=\|f\|_{\infty,\lambda,0,T}+\||f\||_{\gamma,\lambda,0,T}:=\sup_{t\in [0,T]}e^{-\lambda t/2}|f(t)|+\sup_{0 \leq s < t \leq T}e^{-\lambda t/2}\frac{|f(t)-f(s)|}{(t-s)^{\gamma}}.$$
  For simplify, let $
 \|f\|_{\gamma,\lambda}:=\|f\|_{\gamma,\lambda,0,T},\|f\|_{\infty,\lambda}
:=\|f\|_{\infty,\lambda,0,T}$ and $\||f\||_{\gamma,\lambda}:=\||f\||_{\gamma,\lambda,0,T}$.

In what follows we fix $ 0 < \alpha<\gamma<\beta $, $ \frac{1}{2} <\beta <H  $, $ \gamma +\beta >1  $. We will show that for every $ \rho_0 > 0 $ there exists an $ \epsilon_0 > 0 $ so that for $ \epsilon < \epsilon_0 $, $ \lambda > \lambda_0 $ we have
    \begin{equation}\label{Xepsdelta}
    	\|X^\epsilon - \bar{X}\|_{\gamma,\lambda}\leq \rho_0.
    \end{equation}

    Note that the norm here is equivalent to the norm in the conclusion. Here $ \delta\in(0, 1) $ is a parameter depending on $ \rho_0 $. To estimate all the terms in the following inequality we have to consider 3 cases. For the first case the right hand side will be absorbed by the left hand side of the inequality when $ \lambda $ is sufficiently large. The second case includes terms providing estimates like $ C\delta^{2-2\gamma}$ where $ C $ is a priori determined by $ \alpha,\beta,\gamma,T,|x_{0}| $ but independent of $ \rho_0,\lambda,\delta,\epsilon $, then we choose fixed $ \delta $ so that $ C\delta^{2-2\gamma} < \zeta \rho_0 $, $ \zeta > 0 $ sufficiently small. The third case contains terms providing an estimate $ \sqrt{R^{-1}\mathbb{E}[\||B^{H}\||_{\beta}]} $, which can be made arbitrarily small when $R$ is sufficiently large. 

 Let
$
		\mathbf{A}:=\mathbb{E}\big[\|X^{\epsilon}-\bar{X}\|_{\gamma,\lambda}^2\mathbf{1}_{D}\big],$ and divide $ [0, T ] $ into intervals depending of size $ \delta $, where $ \delta\in (0,1) $ is a fixed positive number. For $t \in [k\delta,\min\{(k+1)\delta,T \}]$ and $s(\delta)=\lfloor\frac{s}{\delta}\rfloor\delta$, where $\lfloor\frac{s}{\delta}\rfloor$ is the integer part of $\frac{s}{\delta}$. From (\ref{mvsde1}) and (\ref{avermvsde}), we have
	\begin{eqnarray*}
		\mathbf{A}
		&\leq&
		5 \mathbb{E}\bigg[\bigg\|\int_{0}^{\cdot} ( b({s}/{\epsilon},X_{s}^{\epsilon},\mathscr{L}_{X_{s}^{\epsilon}})-b({s}/{\epsilon},X_{s(\delta)}^{\epsilon},\mathscr{L}_{X_{s(\delta)}^{\epsilon}}) ) \mathrm{d} s\bigg\|_{\gamma,\lambda}^2\mathbf{1}_{D} \bigg]\cr
		& &+
		5\mathbb{E}\bigg[\bigg\|\int_{0}^{\cdot} ( b({s}/{\epsilon},X_{s(\delta)}^{\epsilon},\mathscr{L}_{X_{s(\delta)}^{\epsilon}})-\bar{b}(X_{s(\delta)}^{\epsilon},\mathscr{L}_{X_{s(\delta)}^{\epsilon}}) ) \mathrm{d} s\bigg\|_{\gamma,\lambda}^2\mathbf{1}_{D} \bigg]\cr
		& &+
		5 \mathbb{E}\bigg[\bigg\|\int_{0}^{\cdot} ( \bar{b}(X_{s(\delta)}^{\epsilon},\mathscr{L}_{X_{s(\delta)}^{\epsilon}})-\bar{b}(X_{s}^{\epsilon},\mathscr{L}_{X_{s}^{\epsilon}}) ) \mathrm{d} s\bigg\|_{\gamma,\lambda}^2\mathbf{1}_{D} \bigg]\cr
		& &+
		5 \mathbb{E}\bigg[\bigg\|\int_{0}^{\cdot} ( \bar{b}(X_{s}^{\epsilon},\mathscr{L}_{X_{s}^{\epsilon}})-\bar{b}(\bar{X}_{s},\mathscr{L}_{\bar{X}_{s}}) ) \mathrm{d} s\bigg\|_{\gamma,\lambda}^2\mathbf{1}_{D} \bigg]\cr
		& &+
		5 \mathbb{E}\bigg[\bigg \|\int_{0}^{\cdot}( \sigma(X^\epsilon_{s})-\sigma(\bar{X}_{s}) )\mathrm{d} B^{H}_{s}\bigg\|_{\gamma,\lambda}^2\mathbf{1}_{D} \bigg]
		=:
		\sum_{i= 1}^{5}\mathbf{A}_i.
	\end{eqnarray*}

	By H\"{o}lder’s inequality, it is easy to obtain
	\begin{eqnarray}\label{b1^2}
		\bigg\|\int_{0}^{\cdot}f(s)\mathrm{d}s\bigg\|_{\gamma,\lambda}^2
		\leq
		C_{\gamma,T}\sup_{t\in [0,T]}e^{-\lambda t}\int_{0}^t\frac{|f(s)|^2}{(t-s)^\gamma}\mathrm{d}s.
	\end{eqnarray}
 
    By (H2), (\ref{b1^2}) and Lemma \ref{x-xh}, we obtain
	\begin{eqnarray}\label{a-13}
	\mathbf{A}_1+\mathbf{A}_3
	&\leq& 
	C_{\gamma,T} \mathbb{E}\bigg[\sup_{t\in [0,T]}e^{-\lambda t} \int_{0}^{t} | b({s}/{\epsilon},X_{s}^{\epsilon},\mathscr{L}_{X_{s}^{\epsilon}})-b({s}/{\epsilon},X_{s(\delta)}^{\epsilon},\mathscr{L}_{X_{s(\delta)}^{\epsilon}}) |^2(t-s)^{-\gamma}\mathrm{d} s \mathbf{1}_{D}\bigg]\cr
	& &+
	C_{\gamma,T} \mathbb{E}\bigg[\sup_{t\in [0,T]}e^{-\lambda t} \int_{0}^{t} |\bar{b}(X_{s(\delta)}^{\epsilon},\mathscr{L}_{X_{s(\delta)}^{\epsilon}})-\bar{b}(X_{s}^{\epsilon},\mathscr{L}_{X_{s}^{\epsilon}})|^2(t-s)^{-\gamma}\mathrm{d} s \mathbf{1}_{D}\bigg]\cr
	&\leq& 
	C_{\gamma,T} \mathbb{E}\bigg[\sup_{t\in [0,T]}e^{-\lambda t} \int_{0}^{t}(|X_{s(\delta)}^{\epsilon}-X^\epsilon_s|^{2}+\mathbb{E}[|X_{s(\delta)}^{\epsilon}-X^\epsilon_s|^{2}])(t-s)^{-\gamma}\mathrm{d} s \mathbf{1}_{D}\bigg]\cr
	&\leq& 
	C_{\alpha,\beta,\gamma,T} \mathbb{E}\bigg[\sup_{t\in [0,T]}e^{-\lambda t} \int_{0}^{t}((1+\||B^H\||_{\beta}^2)(1+\|X^{\epsilon}\|_{\gamma}^2)\delta^{2\beta}+\mathbb{E}[(1+\||B^H\||_{\beta}^2)(1+\|X^{\epsilon}\|_{\gamma}^2)\delta^{2\beta}])(t-s)^{-\gamma}\mathrm{d} s \mathbf{1}_{D}\bigg]\cr
	&\leq&
	C_{\alpha,\beta,\gamma,T,|x_{0}|}\delta^{2\beta}.
	\end{eqnarray}
	
	By elementary inequality, we have
	\begin{eqnarray*}
		\mathbf{A}_2
		&\leq &
		C \mathbb{E}\bigg[ \sup_{t\in [0, T]}e^{-\lambda t}\bigg|\int_{0}^{t}( b({s}/{\epsilon}, X_{s(\delta)}^{\epsilon},\mathscr{L}_{X_{s(\delta)}^{\epsilon}})-\bar{b}(X_{s(\delta)}^{\epsilon},\mathscr{L}_{X_{s(\delta)}^{\epsilon}}) ) \mathrm{d}s \bigg|^{2}\mathbf{1}_{D}\bigg]\cr
		& &+ 
		C \mathbb{E}\bigg[ \sup_{0\leq s<t\leq T}e^{-\lambda t}\frac{\big|\int_{s}^{t}( b({r}/{\epsilon}, X_{r(\delta)}^{\epsilon},\mathscr{L}_{X_{r(\delta)}^{\epsilon}})-\bar{b}(X_{r(\delta)}^{\epsilon},\mathscr{L}_{X_{r(\delta)}^{\epsilon}}) ) \mathrm{d} r\big|^2}{(t-s)^{2\gamma}}\mathbf{1}_{D}\bigg]\cr
		&=:& 
	\mathbf{A}_{21}+\mathbf{A}_{22}.
	\end{eqnarray*}	
	
	For $\mathbf{A}_{21}$, by (H1)-(H5), H\"{o}lder inequality and Lemma \ref{xbounded}, we have
	\begin{eqnarray*}
		\mathbf{A}_{21}
		&\leq & 
		C \mathbb{E}\bigg[ \sup_{t\in [0, T]}\bigg|\sum_{k=0}^{\lfloor \frac{t}{\delta}\rfloor-1} \int_{k\delta}^{(k+1)\delta}( b({s}/{\epsilon}, X_{k\delta}^{\epsilon},\mathscr{L}_{X_{k\delta}^{\epsilon}})-\bar{b}(X_{k\delta}^{\epsilon},\mathscr{L}_{X_{k\delta}^{\epsilon}}) ) \mathrm{d}s\bigg|^{2}\bigg]\cr
		& &+
		C \mathbb{E}\bigg[ \sup_{t\in [0, T]}\bigg|\int_{\lfloor \frac{t}{\delta}\rfloor \delta}^{t}( b({s}/{\epsilon}, X_{s(\delta)}^{\epsilon},\mathscr{L}_{X_{s(\delta)}^{\epsilon}})-\bar{b}(X_{s(\delta)}^{\epsilon},\mathscr{L}_{X_{s(\delta)}^{\epsilon}})) \mathrm{d}s \bigg|^{2}\bigg]\cr
		&\leq&
		C \mathbb{E}\bigg[\sup_{t\in [0, T]}\lfloor \frac{t}{\delta}\rfloor \sum_{k=0}^{\lfloor \frac{t}{\delta}\rfloor-1} \bigg|\int_{k\delta}^{(k+1)\delta}( b({s}/{\epsilon}, X_{k\delta}^{\epsilon},\mathscr{L}_{X_{k\delta}^{\epsilon}})-\bar{b}(X_{k\delta}^{\epsilon},\mathscr{L}_{X_{k\delta}^{\epsilon}})) \mathrm{d}s\bigg|^{2}\bigg] \cr
		& &+
		C\mathbb{E}\bigg[\sup_{t\in [0, T]}\Big(t-\lfloor \frac{t}{\delta}\rfloor \delta\Big)\int_{\lfloor \frac{t}{\delta}\rfloor \delta}^{t}(|b({s}/{\epsilon}, X_{s(\delta)}^{\epsilon},\mathscr{L}_{X_{s(\delta)}^{\epsilon}})|^2+|\bar{b}(X_{s(\delta)}^{\epsilon},\mathscr{L}_{X_{s(\delta)}^{\epsilon}})|^2) \mathrm{d}s\Big]\cr
		&\leq & C_{T}\delta^2+
		\frac{C_T}{\delta^2}  \max_{0\leq k \leq \lfloor \frac{T}{\delta}\rfloor-1} \mathbb{E}\bigg[\bigg|\epsilon\int_{\frac{k\delta}{\epsilon}}^{\frac{(k+1)\delta}{\epsilon}}( b(s, X_{k\delta}^{\epsilon},\mathscr{L}_{X_{k\delta}^{\epsilon}})-\bar{b}(X_{k\delta}^{\epsilon},\mathscr{L}_{X_{k\delta}^{\epsilon}}) ) \mathrm{d}s\bigg|^{2}\bigg]
		\cr
		&\leq & 
	C_{T}\delta^2+
		\frac{C_T}{\delta^2}  \max_{0\leq k \leq \lfloor \frac{T}{\delta}\rfloor-1} \mathbb{E}\bigg[\bigg|\frac{T\epsilon}{\delta(k+1)}\int_{0}^{\frac{(k+1)\delta}{\epsilon}}( b(s, X_{k\delta}^{\epsilon},\mathscr{L}_{X_{k\delta}^{\epsilon}})-\bar{b}(X_{k\delta}^{\epsilon},\mathscr{L}_{X_{k\delta}^{\epsilon}}) ) \mathrm{d}s\bigg|^{2}\bigg]\cr
		& &+
		\frac{C_T}{\delta^2}  \max_{0\leq k \leq \lfloor \frac{T}{\delta}\rfloor-1} \mathbb{E}\bigg[\bigg|\frac{T\epsilon}{\delta k}\int_{0}^{\frac{k\delta}{\epsilon}}( b(s, X_{k\delta}^{\epsilon},\mathscr{L}_{X_{k\delta}^{\epsilon}})-\bar{b}(X_{k\delta}^{\epsilon},\mathscr{L}_{X_{k\delta}^{\epsilon}}) ) \mathrm{d}s\bigg|^{2}\bigg].
	\end{eqnarray*}	

We have for $ \epsilon \rightarrow 0 $, $ \frac{\delta(k+1)}{\epsilon} \rightarrow \infty$ for any $ k, 1 \leq k \leq \lfloor \frac{T}{\delta}\rfloor-1 $. In addition we take the maximum over finitely many elements determined by the fixed number $ \delta $ given and $ T $. Following (H5), we have for every element under the maximum
	\begin{align}\label{h33}
\begin{split}
	\max_{0\leq k \leq \lfloor \frac{T}{\delta}\rfloor-1}\bigg|\frac{\epsilon}{\delta(k+1)}&\int_{0}^{\frac{(k+1)\delta}{\epsilon}}( b(s, X_{k\delta}^{\epsilon},\mathscr{L}_{X_{k\delta}^{\epsilon}})-\bar{b}(X_{k\delta}^{\epsilon},\mathscr{L}_{X_{k\delta}^{\epsilon}}) ) \mathrm{d}s\bigg|\\
\leq &	\max_{0\leq k \leq \lfloor \frac{T}{\delta}\rfloor-1} \varphi\bigg(\frac{(k+1)\delta}{\epsilon}\bigg)\big(1+|X_{k\delta}^{\epsilon}|^2+\mathbb{W}_2(\mathscr{L}_{X_{k\delta}^{\epsilon}},\delta_0)\big)\leq C_{\epsilon}
\end{split}
	\end{align}
	where $ C_{\epsilon}\rightarrow0 $, as $ \epsilon\rightarrow0 $. Thus, we have for $ \epsilon $ sufficiently small and the $ \delta $ given
	\begin{equation}\label{a-21}
	\mathbf{A}_{21} \leq C_{\alpha,\beta,\gamma,T,|x_0|}\delta^2,
	\end{equation}
	
	For $\mathbf{A}_{22}$, by (H1)-(H5), H\"{o}lder inequality and Lemma \ref{xbounded} again, we have
	\begin{eqnarray*}
		\mathbf{A}_{22} 
		&\leq &
	C\mathbb{E}\bigg[ \bigg(\sup_{0\leq s<t\leq T}\frac{\big|\int_{s}^{t}( b({r}/{\epsilon}, X_{r(\delta)}^{\epsilon},\mathscr{L}_{X_{r(\delta)}^{\epsilon}})-\bar{b}(X_{r(\delta)}^{\epsilon},\mathscr{L}_{X_{r(\delta)}^{\epsilon}}) ) \mathrm{d} r\big|}{(t-s)^{\gamma}}\bigg)^2 \mathbf{1}_{\ell}\bigg]\cr
		& &+
	C\mathbb{E}\bigg[ \bigg(\sup_{0\leq s<t\leq T}\frac{\big|\int_{s}^{t}( b({r}/{\epsilon}, X_{r(\delta)}^{\epsilon},\mathscr{L}_{X_{r(\delta)}^{\epsilon}})-\bar{b}(X_{r(\delta)}^{\epsilon},\mathscr{L}_{X_{r(\delta)}^{\epsilon}}) ) \mathrm{d} r\big|}{(t-s)^{\gamma}}\bigg)^2 \mathbf{1}_{\ell^c}\bigg]\cr
		&=:&
		\mathbf{A}_{221}+\mathbf{A}_{222}
	\end{eqnarray*}	
	where $\ell:=\{t < (\lfloor \frac{s}{\delta}\rfloor+2)\delta\}$ and $\ell^c:=\{t \geq (\lfloor \frac{s}{\delta}\rfloor+2)\delta\}$. For $\mathbf{A}_{222}$, by (H2), (H3) and the fact  $\ell=\{t < (\lfloor \frac{s}{\delta}\rfloor+2)\delta\}  $ implies that $t -s< \lfloor \frac{s}{\delta}\rfloor \delta -s + 2\delta \leq  2\delta $,  so we have
	\begin{eqnarray*}
		\mathbf{A}_{221}
		\leq
		C\mathbb{E}\Big[\sup_{0\leq s<t\leq T}(t-s)^{2-2\gamma}\mathbf{1}_{\ell}\Big]
		\leq
		C\delta^{2-2\gamma}.
	\end{eqnarray*}
	
	By (H1)-(H5) and the fact that 
	$\lfloor \lambda_1 \rfloor -\lfloor \lambda_2 \rfloor \leq \lambda_1-\lambda_2 +1, $ for $\lambda_1 \geq \lambda_2 \geq 0$, we have
	\begin{eqnarray*}
		\mathbf{A}_{222}
		&\leq&
		C \mathbb{E}\bigg[ \sup_{0\leq s<t\leq T} \frac{\big|\int_{s}^{(\lfloor \frac{s}{\delta}\rfloor+1)\delta}( b({r}/{\epsilon}, X_{r(\delta)}^{\epsilon},\mathscr{L}_{ X_{r(\delta)}^{\epsilon}})-\bar{b}( X_{r(\delta)}^{\epsilon},\mathscr{L}_{X_{r(\delta)}^{\epsilon}}) ) \mathrm{d}r\big|^{2}}{(t-s)^{2\gamma}} \mathbf{1}_{\ell^c}\bigg]\cr
		& &+
		C \mathbb{E}\bigg[ \sup_{0\leq s<t\leq T} \frac{\big|\int_{\lfloor \frac{t}{\delta}\rfloor\delta}^{t}( b({r}/{\epsilon}, X_{r(\delta)}^{\epsilon},\mathscr{L}_{ X_{r(\delta)}^{\epsilon}})-\bar{b}( X_{r(\delta)}^{\epsilon},\mathscr{L}_{X_{r(\delta)}^{\epsilon}}) ) \mathrm{d}r\big|^{2}}{(t-s)^{2\gamma}} \mathbf{1}_{\ell^c}\bigg]\cr
		& &+ 
		C \mathbb{E}\bigg[ \sup_{0\leq s<t\leq T} \frac{\big|\sum_{k=\lfloor \frac{s}{\delta}\rfloor +1}^{\lfloor \frac{t}{\delta}\rfloor-1} \int_{k\delta}^{(k+1)\delta }( b({r}/{\epsilon}, X_{k\delta}^{\epsilon},\mathscr{L}_{ X_{k\delta}^{\epsilon}})-\bar{b}(X_{k\delta}^{\epsilon},\mathscr{L}_{X_{k\delta}^{\epsilon}}) ) \mathrm{d}r\big|^{2}}{(t-s)^{2\gamma}}\mathbf{1}_{\ell^c}\bigg]\cr
		&\leq& 
	C\mathbb{E}\bigg[ \sup_{0\leq s<t\leq T} \frac{\big|\int_{s}^{(\lfloor \frac{s}{\delta}\rfloor+1)\delta}( b({r}/{\epsilon}, X_{r(\delta)}^{\epsilon},\mathscr{L}_{ X_{r(\delta)}^{\epsilon}})-\bar{b}( X_{r(\delta)}^{\epsilon},\mathscr{L}_{X_{r(\delta)}^{\epsilon}}) ) \mathrm{d}r\big|^{2}}{(t-s)^{2\gamma}} \mathbf{1}_{\ell^c}\bigg]\cr
		& &+
		C\mathbb{E}\bigg[ \sup_{0\leq s<t\leq T} \frac{\big|\int_{\lfloor \frac{t}{\delta}\rfloor\delta}^{t}( b({r}/{\epsilon}, X_{r(\delta)}^{\epsilon},\mathscr{L}_{ X_{r(\delta)}^{\epsilon}})-\bar{b}( X_{r(\delta)}^{\epsilon},\mathscr{L}_{X_{r(\delta)}^{\epsilon}}) ) \mathrm{d}r\big|^{2}}{(t-s)^{2\gamma}} \mathbf{1}_{\ell^c}\bigg]\cr
		& &+ 
		C\mathbb{E}\bigg[ \sup_{0\leq s<t\leq T} \frac{(\lfloor \frac{t}{\delta}\rfloor-\lfloor \frac{s}{\delta}\rfloor-1)}{(t-s)^{2\gamma}}\sum_{k=\lfloor \frac{s}{\delta}\rfloor +1}^{\lfloor \frac{t}{\delta}\rfloor-1} \bigg|\int_{k\delta}^{(k+1)\delta }( b({r}/{\epsilon}, X_{k\delta}^{\epsilon},\mathscr{L}_{ X_{k\delta}^{\epsilon}})-\bar{b}(X_{k\delta}^{\epsilon},\mathscr{L}_{X_{k\delta}^{\epsilon}})) \mathrm{d}r\bigg|^{2}\mathbf{1}_{\ell^c}\bigg]\cr
		&\leq&
		C\delta^{2-2\gamma}
		+ \frac{C_{\gamma,T}}{\delta^2}   \max_{0\leq k \leq\lfloor \frac{T}{\delta}\rfloor-1}\mathbb{E}\bigg[\bigg|\int_{k\delta}^{(k+1)\delta }( b({r}/{\epsilon}, X_{k\delta}^{\epsilon},\mathscr{L}_{ X_{k\delta}^{\epsilon}})-\bar{b}(X_{k\delta}^{\epsilon},\mathscr{L}_{X_{k\delta}^{\epsilon}}) ) \mathrm{d}r\bigg|^{2} \mathbf{1}_{\ell^c}\bigg].
	\end{eqnarray*}
	
	Using (H5) again, the remaining term on the right hand side can be estimated similar to $ \mathbf{A}_{21} $, see (\ref{h33}). We have
	\begin{equation}\label{a-22}
	\mathbf{A}_{22} \leq C_{\alpha,\beta,\gamma,T,|x_0|}\delta^{2-2\gamma}.
	\end{equation}
		
		For $ \mathbf{A}_4 $, by (\ref{b1^2}), we have
		\begin{eqnarray}\label{a-4}
		\mathbf{A}_4 
		&\leq& 
		C_{\gamma,T}\mathbb{E}\bigg[\sup_{t\in [0,T]} e^{-\lambda t}\int_{0}^{t} |\bar{b}(X_{s}^{\epsilon},\mathscr{L}_{X_{s}^{\epsilon}})-\bar{b}(\bar{X}_{s},\mathscr{L}_{\bar{X}_{s}}) |^2(t-s)^{-\gamma}\mathbf{1}_{D}\mathrm{d} s\bigg]\cr
		&\leq& 
		C_{\gamma,T}\mathbb{E}\bigg[\sup_{t\in [0,T]} \int_{0}^{t}e^{-\lambda (t-s)}e^{-\lambda s}(t-s)^{-\gamma} |\bar{b}(X_{s}^{\epsilon},\mathscr{L}_{X_{s}^{\epsilon}})-\bar{b}(\bar{X}_{s},\mathscr{L}_{\bar{X}_{s}})|^2\mathbf{1}_{D}\mathrm{d} s\bigg]\cr
		&\leq & 
		C_{\gamma,T}\mathbb{E}\bigg[\sup_{t\in [0,T]}\int_{0}^{t} e^{-\lambda (t-s)}(t-s)^{-\gamma} e^{-\lambda s}\big( |X^\epsilon_{s}-\bar{X}_{s}|^2+\mathbb{E}[|X^\epsilon_{s}-\bar{X}_{s}|^2] \big)\mathbf{1}_{D}\mathrm{d} s\bigg]\cr
		&\leq& 
		C_{\gamma,T}\mathbb{E}\big[\sup_{t\in [0,T]} e^{-\lambda t}\big( |X^\epsilon_{t}-\bar{X}_{t}|^2+\mathbb{E}[|X^\epsilon_{t}-\bar{X}_{t}|^2] \big)\mathbf{1}_{D}\big] \sup_{t\in [0, T]} \int_{0}^{t} e^{-\lambda (t-s)}(t-s)^{-\gamma} \mathrm{d} s\cr
		&\leq& 
		C_{\gamma,T} \lambda^{\gamma-1} \mathbb{E}\big[\|X^\epsilon-\bar{X}\|_{\gamma,\lambda}^2\mathbf{1}_{D}\big]
		+C_{\gamma,T} \mathbb{E}\big[\|X^\epsilon-\bar{X}\|_{\gamma,\lambda}^2\mathbf{1}_{D^c}\big].
		\end{eqnarray}
	
	To proceed, we have
	\begin{eqnarray*}
		\mathbf{A}_5
	&	=&
		\mathbb{E}\bigg[\sup_{t\in [0,T]}e^{-\lambda t}\bigg|\int_{0}^{t}( \sigma(X^\epsilon_{s})-\sigma(\bar{X}_{s}) )\mathrm{d} B^{H}_{s}\bigg|^2\bigg]
		+\mathbb{E}\bigg[\sup_{0 \leq s < t \leq T}\frac{e^{-\lambda t}}{(t-s)^{2\gamma}}\bigg|\int_{s}^{t}( \sigma(X^\epsilon_{s})-\sigma(\bar{X}_{s}) )\mathrm{d} B^{H}_{s}\bigg|^2\bigg]\cr
&=:&\mathbf{A}_{51}+\mathbf{A}_{52}.
	\end{eqnarray*}
	
	Since (\ref{grsinteg}) and Lemma \ref{lemma2.5}, we have
	\begin{align}\label{a5esti}
		&e^{-\lambda t}\bigg|\int_{s}^{t}( \sigma(X^\epsilon_{r})-\sigma(\bar{X}_{r}) )\mathrm{d} B^{H}_{r}\bigg|^2\cr
		\leq&
		C_{\alpha,\beta}e^{-\lambda t}\bigg(\int_{s}^{t}\bigg(\frac{|\sigma(X^\epsilon_{r})-\sigma(\bar{X}_{r})|}{(r-s)^{\alpha}}
		+
		\int_{s}^{r}\frac{|\sigma(X^\epsilon_{r})-\sigma(\bar{X}_{r})-\sigma(X^\epsilon_{u})+\sigma(\bar{X}_{u})|}{(r-u)^{\alpha+1}}\mathrm{d}u\bigg)\frac{\||B^H\||_{\beta}}{(t-r)^{-\alpha-\beta+1}}\mathrm{d}r\bigg)^2\cr
		\leq&
		C_{\alpha,\beta}\||B^H\||_{\beta}^2\bigg(\int_{s}^{t}e^{-\frac{\lambda (t-r)}{2}}\frac{e^{-\frac{\lambda r}{2}}|\sigma(X^\epsilon_{r})-\sigma(\bar{X}_{r})|}{(r-s)^{\alpha}}(t-r)^{\alpha+\beta-1}\mathrm{d}r\cr
		&+
		\int_{s}^{t}\int_{s}^{r}e^{-\frac{\lambda (t-r)}{2}}\frac{e^{-\frac{\lambda r}{2}}|\sigma(X^\epsilon_{r})-\sigma(\bar{X}_{r})-\sigma(X^\epsilon_{u})+\sigma(\bar{X}_{u})|(r-u)^{\gamma}}{(r-u)^{\gamma}(r-u)^{\alpha+1}}\mathrm{d}u(t-r)^{\alpha+\beta-1}\mathrm{d}r\bigg)^2\cr
		\leq&
		C_{\alpha,\beta}\||B^H\||_{\beta}^2\|X^\epsilon-\bar{X}\|_{\gamma,\lambda}^2\bigg(\int_{s}^{t}e^{-\lambda (t-r)}(r-s)^{-\alpha}(t-r)^{\alpha-1}\mathrm{d}r\bigg)\bigg(\int_{s}^{t}(r-s)^{-\alpha}(t-r)^{\alpha+2\beta-1}\mathrm{d}r\bigg)\cr
		&+
		C_{\alpha,\beta,\gamma}\||B^H\||_{\beta}^2\|X^\epsilon-\bar{X}\|_{\gamma,\lambda}^2(1+\|X^\epsilon\|_{\gamma}^2+\|\bar{X}\|_{\gamma}^2)\bigg(\int_{s}^{t}e^{-\lambda (t-r)}(r-s)^{-\alpha}(t-r)^{\alpha-1}\mathrm{d}r\bigg)\bigg(\int_{s}^{t}(r-s)^{2\gamma-\alpha}(t-r)^{\alpha+2\beta-1}\mathrm{d}r\bigg)\cr
		\leq&
		C_{\alpha,\beta}\||B^H\||_{\beta}^2(t-s)^{2\beta}\|X^\epsilon-\bar{X}\|_{\gamma,\lambda}^2\int_{s}^{t}e^{-\lambda (t-r)}(r-s)^{-\alpha}(t-r)^{\alpha-1}\mathrm{d}r\cr
		&+
		C_{\alpha,\beta,\gamma}\||B^H\||_{\beta}^2(t-s)^{2(\beta+\gamma)}\|X^\epsilon-\bar{X}\|_{\gamma,\lambda}^2(1+\|X^\epsilon\|_{\gamma}^2+\|\bar{X}\|_{\gamma}^2)\int_{s}^{t}e^{-\lambda (t-r)}(r-s)^{-\alpha}(t-r)^{\alpha-1}\mathrm{d}r
	\end{align}
	where Lemma 7.1 in \cite{rascanu2002differential} implies that
	\begin{align*}
	|\sigma(X^\epsilon_{r})-\sigma(\bar{X}_{r})-\sigma(X^\epsilon_{u})+\sigma(\bar{X}_{u})|
	\leq
	|X^\epsilon_{r}-\bar{X}_{r}-X^\epsilon_{u}+\bar{X}_{u}|
	+|X^\epsilon_{r}-\bar{X}_{r}|(|X^\epsilon_{r}-X^\epsilon_{u}|+|\bar{X}_{r}-\bar{X}_{u}|)
	\end{align*}
and by a change of variable $ v=\frac{r-s}{t-s} $, from Lemma 8 in \cite{chen2013pathwise} and the fact that $ \gamma<\beta $, it is easy to see that
	\begin{align*}
		&(t-s)^{2\beta}\int_{s}^{t}e^{-\lambda (t-r)}(r-s)^{-\alpha}(t-r)^{\alpha-1}\mathrm{d}r\cr
		&=
		(t-s)^{2\gamma}(t-s)^{2(\beta-\gamma)}\int_{0}^{1}e^{-\lambda (t-s)(1-v)}v^{-\alpha}(1-v)^{\alpha-1}\mathrm{d}v
		\leq
		(t-s)^{2\gamma}K(\lambda)
	\end{align*}
where $ K(\lambda)\rightarrow 0 $ as $ \lambda \rightarrow \infty$.

	Then, by Lemma \ref{xbounded}, we have
	\begin{eqnarray}
		\mathbf{A}_{52}
		\leq
		C_{\alpha,\beta,\gamma,T}K(\lambda)\mathbb{E}[\||B^H\||_{\beta}^2\|X^\epsilon-\bar{X}\|_{\gamma,\lambda}^2(1+\|X^\epsilon\|_{\gamma}^2+\|\bar{X}\|_{\gamma}^2)\mathbf{1}_{D}]
		\leq
		C_{\alpha,\beta,\gamma,T,R}K(\lambda)\mathbb{E}\big[\|X^\epsilon-\bar{X}\|_{\gamma,\lambda}^2\mathbf{1}_{D}\big].
	\end{eqnarray}
	
	In a similar manner than before for the first expression on $ \mathbf{A}_{52} $, we obtain
	\begin{eqnarray*}
	\mathbf{A}_{51}
	\leq
	C_{\alpha,\beta,\gamma,T,R}K(\lambda)\mathbb{E}\big[\|X^\epsilon-\bar{X}\|_{\gamma,\lambda}^2\mathbf{1}_{D}\big].
	\end{eqnarray*}
	
	Thus, we have
	\begin{eqnarray}\label{a-5}
		\mathbf{A}_{5}
		\leq
		C_{\alpha,\beta,\gamma,T,R}K(\lambda)\mathbb{E}\big[\|X^\epsilon-\bar{X}\|_{\gamma,\lambda}^2\mathbf{1}_{D}\big].
	\end{eqnarray}

	 Summing up (\ref{a-13}), (\ref{a-21}), (\ref{a-22}), (\ref{a-4}) and (\ref{a-5}) and the fact that $\mathbb{P}\big\{\tau_R < T\big\} \leq R^{-1}\mathbb{E}[\||B^{H}\||_{\beta}]$ (see Lemma 4.7 in  \citeauthor{pei2020averaging}, \citeyear{pei2020averaging}), we obtain
	\begin{eqnarray*}
		\mathbf{A}
		&\leq&
        C_{\alpha,\beta,\gamma,T,|x_{0}|}\delta^{2\beta}
        +C_{\alpha,\beta,\gamma,T,|x_0|}\delta^2+C_{\alpha,\beta,\gamma,T,|x_0|}\delta^{2-2\gamma}
        +C_{\alpha,\beta,\gamma,T,R}K(\lambda)\mathbb{E}\big[\|X^\epsilon-\bar{X}\|_{\gamma,\lambda}^2\big]\cr
        & &+
        C_{\gamma,T} \lambda^{\gamma-1} \mathbb{E}\big[\|X^\epsilon-\bar{X}\|_{\gamma,\lambda}^2\mathbf{1}_{D}\big]
        +C_{\gamma,T} \lambda^{\gamma-1} \mathbb{E}\big[\|X^\epsilon-\bar{X}\|_{\gamma,\lambda}^2\mathbf{1}_{D^c}\big]\cr
		&\leq&
		C_{\alpha,\beta,\gamma,T,|x_{0}|}\delta^{2-2\gamma}
		+C_{\alpha,\beta,\gamma,T,R,|x_0|}\big(\lambda^{\gamma-1}+K(\lambda)\big) \mathbb{E}\big[\|X^\epsilon-\bar{X}\|_{\gamma,\lambda}^2\mathbf{1}_{D}\big]+
		C_{\gamma,T}\lambda^{\gamma-1}\sqrt{R^{-1}\mathbb{E}[\||B^{H}\||_{\beta}]}.
	\end{eqnarray*}
	
	Taking $\lambda$ large enough, such that $C_{\alpha,\beta,\gamma,T,R,|x_0|}\big(\lambda^{\gamma-1}+K(\lambda)\big)\vee C_{\gamma,T}\lambda^{\gamma-1}<1$, we have
	\begin{eqnarray*}
	\mathbb{E}\big[\|X^\epsilon-\bar{X}\|_{\gamma}^2\mathbf{1}_{D}\big]
	\leq 
	C_{\alpha,\beta,\gamma,T,|x_{0}|}e^{\lambda T}\delta^{2-2\gamma}
	+
	e^{\lambda T}\sqrt{R^{-1}\mathbb{E}[\||B^{H}\||_{\beta}]}.
	\end{eqnarray*}
	
	Next, we return to the second supremum on the right-hand side of inequality (\ref{x-xbarstop}), by Cauchy-Schwarz's inequality, Lemma \ref{xbounded} and using Lemma 4.7 in  \cite{pei2020averaging} again, we have
	\begin{eqnarray*}
	\mathbb{E}\big[\|X^{\epsilon}-\bar{X}\|_{\gamma}^2\mathbf{1}_{\{\tau_R < T\}}\big] 
	\leq 
	\Big(\mathbb{E}\big[\|X^{\epsilon}-\bar{X}\|_{\gamma}^4\big]\Big)^{1/2}\mathbb{P}\big\{\tau_R < T\big\}^{1/2}
	\leq
	C_{\alpha,\beta,\gamma,T,|x_0|}\sqrt{R^{-1}\mathbb{E}[\||B^{H}\||_{\beta}]}.
	\end{eqnarray*}
	
	Summing above and let $R \rightarrow \infty$, we have 
	\begin{eqnarray*}
		\lim_{\epsilon\rightarrow0}\mathbb{E}\big[\|X^{\epsilon}-\bar{X}\|_{\gamma}^2\big]  =0.
	\end{eqnarray*}
	This completes the proof.\qed

\section*{Acknowledgments}
This work was partially supported by National Natural Science Foundation of China (NSF) under Grant No. 12172285, NSF of Chongqing under Grant No.cstc2021jcyj-msxmX0296, Shaanxi Fundamental Science Research Project for Mathematics and Physics under Grant No. 22JSQ027 and Fundamental Research Funds for the Central Universities.

\bibliographystyle{model4-names}
\biboptions{authoryear}
\bibliography{reference2}







\end{document}